\documentclass[a4paper,12pt]{article}

\usepackage{amsmath}
\usepackage{amsfonts} 
\usepackage{amssymb,amsthm}
\usepackage{a4wide}
\usepackage[english]{babel}   

\newtheorem{de}{Definition}[section]

\newtheorem{prop}[de]{Proposition}
\newtheorem{cor}[de]{Corollary}

\newtheorem{example}[de]{Example}

\begin{document}

\title{Complex supermanifolds of odd dimension beyond 5}
\author{
{\bf Matthias Kalus}\\
{\small  Fakult\"at f\"ur Mathematik}\\{\small Ruhr-Universit\"at Bochum}\\ {\small D-44780 Bochum, Germany}\\
}
\date{}
\maketitle
\begin{abstract}\noindent
Any non-split complex supermanifold is a deformation of a split supermanifold. These deformations are classified by group orbits in a non-abelian cohomology. For the case of a split supermanifold with no global nilpotent even vector fields, an injection of this non-abelian cohomology into an abelian cohomology is constructed. The cochains in the non-abelian complex appear as exponentials of cochains of nilpotent even derivations. Necessary conditions for a recursive construction of these cochains of derivations are analyzed up to terms of degree six. Results on classes of examples of supermanifolds of odd dimension up to 7 are deduced.
\end{abstract}

Complex supermanifolds appear as deformations of split complex supermanifolds $(M,\mathcal O_{\Lambda E})$, where $E\to M$ is a complex vector bundle and $\mathcal O_{\Lambda E}$ denotes the sheaf of holomorphic sections of $\Lambda E$. These deformations of a split complex supermanifold can be parametrized by $H^0(M,Aut(E))$-orbits in a non-abelian first \v{C}ech cohomology $H^1(M,G_E)$ (see \cite{Gr}). The cocycles of this cohomology appear as exponentials of nilpotent derivations $u$ in $C^1(M,Der^{(2)}(\Lambda E))$ (see \cite{Ro}). Here $Der^{(2)}(\Lambda E)$ denotes the even derivations of $\mathcal O_{\Lambda E}$ that increase the degree by at least two.  In detail, the cochain $u$ has to satisfy the non-abelian cocycle condition $\mathbf{d}\exp(u):=(\exp(u_{ij})\exp(u_{jk})\exp(u_{ki}))_{ijk}=Id$. Due to nilpotency, the appearing exponential series are finite and their length increases with the rank of $E$. So the non-abelian cocycle condition on $u$ becomes more and more complicated with higher odd 
dimension.

\bigskip\noindent
A naturally arising computational question is how to find suitable $u$ that yield supermanifold structures.  We are aiming at the questions:
\begin{itemize}
 \item[(A)] Is it possible to express the non-abelian cocycle condition on $u$ up to non-abelian coboundaries as conditions in the abelian cohomology $H^1(M,Der^{(2)}(\Lambda E))$?
 \end{itemize}
 The $\mathbb Z$-grading of $\mathcal O_{\Lambda E}$ induces a $\mathbb Z$-grading on $Der^{(2)}(\Lambda E)$. So $u$ is a finite sum $u_2+u_4+u_6+\ldots$. Let $2\leq q\leq rank(E)$.
 \begin{itemize}
 \item[(B)] What are the necessary and sufficient conditions on a sum $u_2+\cdots+u_{2q-2}$ to be extendable to a $u\in C^1(M,Der^{(2)}(\Lambda E))$ that defines a supermanifold structure?
\end{itemize}

\bigskip\noindent
In the first section we answer Question (A) for split complex supermanifolds with no global even vector fields that increase the degree by two or more. Speaking of automorphisms of the split supermanifold, this condition $H^0(M,Der^{(2)}(\Lambda E))=0$ can be reformulated as follows: there is no automorphism whose degree preserving part is the identity but the identity itself. Under this condition we construct a well-defined injection $\sigma_D:H^1(M,G_E)\to H^1(M,Der^{(2)}(\Lambda E))$. This generalizes a result on supermanifolds of odd dimension up to $5$ in \cite{Ka}. Note at this point that differing methods for determining the cohomology $H^1(M,G_E)$ can be found in \cite{On}.

\bigskip\noindent
For Question $B$, assume that $u_{(2q-2)}=u_2+\cdots+u_{2q-2}$ satisfies the non-abelian cocycle condition $\mathbf{d}\exp(u_{(2q-2)})$ up to terms of degree $2q$ and higher. The necessary and sufficient condition for the existence of a $u_{(2q)}:=u_{(2q-2)}+u_{2q}$ satisfying $\mathbf{d}\exp(u_{(2q)})$ up to terms of degree $2q+2$ and higher is $pr_{2q}\mathbf{d}\exp(u_{(2q-2)})\in B^2(M,Der_{2q}(\Lambda E))$. Here $pr_{2q}$ denotes the projection onto the degree $2q$ component. In general it is not at all clear that $pr_{2q}\mathbf{d}\exp(u_{(2q-2)})$  either lies in $Z^2(M,End_{2q}(\Lambda E))$ or in  $C^2(M,Der_{2q}(\Lambda E))$. However, we show in the second section that for $q=2,3$, the condition $pr_{2q}\mathbf{d}\exp(u_{(2q-2)})\in Z^2(M,End_{2q}(\Lambda E))$ is automatically satisfied. Even better, the condition $pr_{6}\mathbf{d}\exp(u_{(4)})\in Z^2(M,Der_{6}(\Lambda E))$ only depends on $u_2$. This yields results for several classes of examples where the cohomology $H^2(M,Der_{2q}(\Lambda E))$ 
vanishes.

\bigskip\noindent
We fix the notation. Let $E \to M$ be a holomorphic vector bundle on a complex manifold $M$. Denote by $\mathcal O_E$ its sheaf of sections, by $\mathcal O_{\Lambda E}$ the associated exterior algebra, by $Aut({\Lambda E})$ the sheaf of automorphisms of the $\mathbb Z/2\mathbb Z$-graded sheaf of algebras $\mathcal O_{\Lambda E}$ and by $Der({\Lambda E})$ and $End({\Lambda E})$ the sheaves of even $\mathbb C$-linear derivations, resp.~endomorphisms of $\mathcal O_{\Lambda E}$. Note that the last two sheaves carry a natural $2\mathbb Z$-grading $Der_{2k}({\Lambda E})$, resp. $End_{2k}(\Lambda E)$ given by the condition  $u(\mathcal O_{\Lambda^jE})\subset \mathcal O_{\Lambda^{j+2k}E}$ for all $j \geq 0$. Furthermore set for $k \geq 0$ 
\begin{align*}
  &Der^{(2k)}(\Lambda E):={\textstyle \bigoplus}_{\ell=k}^\infty Der_{2 \ell}(\Lambda E)\ , &&\quad End^{(2k)}(\Lambda E):={\textstyle \bigoplus}_{\ell=k}^\infty End_{2 \ell}(\Lambda E)\ , \\
   &Der_{(2k)}(\Lambda E):={\textstyle \bigoplus}_{\ell=1}^k Der_{2 \ell}(\Lambda E)\qquad  \mbox{ and } &&\quad End_{(2k)}(\Lambda E):={\textstyle \bigoplus}_{\ell=1}^k End_{2 \ell}(\Lambda E) \ .
\end{align*}
Denote for $k\geq 1$ the induced projections by $pr_{2k}:End^{(2)}(\Lambda E) \to End_{2k}(\Lambda E)$ and further  $pr_{(2k)}:End^{(2)}(\Lambda E) \to  End_{(2k)}(\Lambda E)$. Let $G_E \subset Aut({\Lambda E})$ denote the subsheaf of automorphisms satisfying 
$(\varphi-Id)(\mathcal O_{\Lambda^jE})\subset {\textstyle\bigoplus}_{k\geq 1} \mathcal O_{\Lambda^{j+2k}E} \quad \forall j \geq 0$. It was shown in \cite{Ro}, that the exponential  $\exp:End({\Lambda E}) \to Aut({\Lambda E})$ yields a bijection between $Der^{(2)}(\Lambda E)$ and $G_E$. We will frequently use that $pr_{(2q)}\circ f \circ pr_{(2q)}=pr_{(2q)}\circ f$ for $f=\exp$ or $\log$. In the following $\mathbf{d}:C^1(M,Aut(\Lambda E)) \to C^2(M,Aut(\Lambda E))$ denotes the coboundary map with respect to composition. In contrast denote by $d:C^1(M,End^{(2)}(\Lambda E)) \to C^2(M,End^{(2)}(\Lambda E))$ the coboundary map with respect to addition.

\bigskip\noindent
Starting on the other hand with a complex supermanifold $\mathcal M=(M,\mathcal O_\mathcal M)$, the nilpotent elements $\mathcal O_\mathcal M^{nil}$ define the locally free $\mathcal O_M$-module $\mathcal O_\mathcal M^{nil}/(\mathcal O_\mathcal M^{nil})^2$ yielding a holomorphic vector bundle $E$. Denote by $Aut(E)$ the sheaf of automorphisms of the vector bundle $E$ over the identity. It is shown in \cite{Gr} that the isomorphism classes of complex supermanifolds associated in this way with a fixed vector bundle $E$ are parametrized by the $H^0(M,Aut(E))$-orbits on the \v{C}ech cohomology $H^1(M,G_E)$. Note that this cohomology is defined with respect to the composition of maps and hence non-abelian in general.

\section{Embedding non-abelian in abelian cohomology}
\noindent
Aiming at an embedding of the cohomology $H^1(M,G_E)$ into the abelian cohomology of sheaves of $\mathcal O_M$-modules $H^1(M,Der^{(2)}(\Lambda E))$, define for $q\geq 2$ the maps:
\begin{align*}
\begin{array}{rccc}
 R_{2q}:\ &C^1(M, Der^{(2)}(\Lambda E)) \quad & \longrightarrow & C^2(M, End_{2q}(\Lambda E))\\
& u &\longmapsto &  pr_{2q}\big( \mathbf{d}\exp( pr_{(2q-2)}(u))\big)\end{array}
\end{align*}
Denote: 
$$\tilde C^1(M, Der^{(2)}(\Lambda E)):=\{u \in C^1(M, Der^{(2)}(\Lambda E)) \ | \ \exp(u) \in Z^1(M,G_E) \}$$
Note that $pr_{2q}\big( \mathbf{d}\exp(u)\big)=d(pr_{2q}u)+ pr_{2q}(\mathbf{d}\exp( pr_{(2q-2)}(u)))$, so $\exp(u) \in Z^1(M,G_E)$ is equivalent to $ R_{2q}(u)=-d(pr_{2q}(u))$ for all $q\geq 2$ and $d(pr_2(u))=0$. Hence  the images of $\tilde C^1(M, Der^{(2)}(\Lambda E))$ under the maps $R_{2q}$  lie in $ B^2(M, Der_{2q}(\Lambda E))$, respectively. Note that the maps $R_{2q}$ only depend on the component in $Der_{(2q-2)}(\Lambda E)$ of the argument. Hence it is possible to choose maps: 
\begin{align*}
D_{2q}^\prime:  \tilde C^1(M, Der^{(2)}(\Lambda E))  \to C^1(M, Der_{2q}(\Lambda E)) \qquad \mbox{ for }q\geq 2
\end{align*}
that factorize over $ pr_{(2q-2)}:\tilde C^1(M, Der^{(2)}(\Lambda E))\to C^1(M, Der^{(2)}(\Lambda E))$ and 
satisfy the equation $dD^\prime_{2q}(u)=R_{2q}(u)$ for all $u \in  \tilde C^1(M, Der^{(2)}(\Lambda E))$. 
Now set for $q\geq 2$:\footnote{We denote: $(exp(v).exp(u))_{ij}=\exp(v_i)\exp(u_{ij})\exp(-v_j)$}
\begin{align*}
 F_{2q}: &C^0(M, Der^{(2)}(\Lambda E))\times C^1(M, Der^{(2)}(\Lambda E))  \longrightarrow  C^1(M, End^{(2)}(\Lambda E))\\
 & \qquad\qquad (v,u) \longmapsto  pr_{(2q)}( \exp(pr_{(2q-2)}(v)).\exp(pr_{(2q-2)}(u)))
\end{align*}
Set $\lambda(v,u):=\log(\exp(v).\exp(u))$ for $(v,u)\in C^0(M, Der^{(2)}(\Lambda E))\times C^1(M, Der^{(2)}(\Lambda E))$ and note that $u\in \tilde C^1(M,Der^{(2)}(\Lambda E))$ includes  $\lambda(v,u)\in \tilde C^1(M,Der^{(2)}(\Lambda E))$. 
Continue the $D_{2q}^\prime$, $q\geq 2$  to maps
\begin{align*}
 &D_{2q}: \ H^0 (M,Aut(E))\times C^0(M, Der^{(2)}(\Lambda E))\times \tilde C^1(M, Der^{(2)}(\Lambda E))  \\ &\qquad\qquad\qquad\qquad\qquad\qquad\qquad\qquad\qquad\qquad\qquad\qquad\qquad\longrightarrow  C^1(M, Der_{2q}(\Lambda E))
\end{align*}
via $D_{2q}(Id,0,u):=D_{2q}^\prime(u)$ and:
\begin{align}\label{eq:8}
\varphi.D_{2q}(\varphi,v,u):=D_{2q}(Id,0,\varphi.\lambda(v,u))+pr_{2q}(\log(F_{2q}(\varphi.v,\varphi.u)))
\end{align}
and set $D=\sum_{q=2}^\infty D_{2q}$.

\begin{prop}\label{theo}
If there is a choice of maps $D^\prime_{2q}$, $q\geq 2$ satisfying: 
\begin{align}\label{eqF}
D^\prime_{2q}(\lambda(v,u))=D^\prime_{2q}(u)-pr_{2q}(\log(F_{2q}(v,u)))
\end{align}
for all $(v,u)\in C^0(M, Der^{(2)}(\Lambda E))\times \tilde C^1(M, Der^{(2)}(\Lambda E))$ then the induced map 
\begin{align*}
 \sigma_D: \ H^1(M,G_E) \longmapsto H^1(M,Der^{(2)}(\Lambda E))
\end{align*}
given by $\sigma_D([\exp(u)])=[D(Id,0,u)+u]$ is well-defined and injective. If additionally the $D^\prime_{2q}$ can be chosen to be $H^0(M,Aut(E))$-equivariant, then $\sigma_D$ is $H^0(M,Aut(E))$-equivariant.
\end{prop}

\begin{proof}
For $\exp(u) \in Z^1(M,G_E)$ it is $dD_{2q}(Id,0,u)=dD^\prime_{2q}(u)=R_{2q}(u)=-d(pr_{2q}(u))$ for all $q\geq 2$ and $dpr_2(u)=0$. So we find  $D(Id,0,u)+u\in Z^1(M,Der^{(2)}(\Lambda E))$. 
Further note that:
\begin{align*}
 pr_{(2q)}(\exp(v).\exp(u))=dpr_{2q}(v)+pr_{2q}(u)+F_{2q}(v,u)
\end{align*}
Using this, $H^0(M,Aut(E))$-equivariance of $\lambda$ and $F_{2q}$, and reasons of degree:
\begin{align*}
 pr_{2q}(\varphi.\lambda(v,u))&= pr_{2q}\log( pr_{(2q)}(\exp(\varphi.v).\exp(\varphi.u)))\\
 &= pr_{2q}\log(dpr_{2q}(\varphi.v)+pr_{2q}(\varphi.u)+F_{2q}(\varphi.v,\varphi.u))\\
 &=pr_{2q}(d\varphi.v+\varphi.u)+pr_{2q}(\log(F_{2q}(\varphi.v,\varphi.u)))
\end{align*}
So:
\begin{align*}
 &\sigma_D([\varphi.(\exp(v).\exp(u))])=\sigma_D([\exp(\varphi.\lambda(v,u))])\\=\  &[D(Id,0,\varphi.\lambda(v,u))+d\varphi.v+\varphi.u+\textstyle{\sum_{q=2}^\infty} pr_{2q}(\log(F_{2q}(\varphi.v,\varphi.u))])
\end{align*}
differing from $\varphi.\sigma_D([\exp(u)])$ with (\ref{eq:8}) and (\ref{eqF}) by $d\varphi.v+D(Id,0,\varphi.u)-\varphi.D(Id,0,u)$. For $\varphi=Id$, the first part of the Proposition follows. The second statement follows for  $H^0(M,Aut(E))$-equivariant  $D^\prime_{2q}$.
\end{proof}

\begin{cor}\label{cor}
If $H^0(M,Der^{(2)}(\Lambda E))=0$ then there is a $D$ such that $\sigma_D$ is well-defined and injective. 
\end{cor}

\begin{proof}
 We show that there exists a choice of the $D^\prime_{2q}$ satisfying (\ref{eqF}). 
 First we have to check that $D^\prime_{2q}$ can be well-defined as a map satisfying (\ref{eqF}). That is that $\lambda(v,u)=u$ includes $pr_{2q}(\log(F_{2q}(v,u)))=0$ for all $q \geq 2$. 
 We follow by induction that $\lambda(v,u)=u$ includes $v=0$: $pr_0(v)=0$. Assume that $pr_{2s}(v)=0$ for all $s< q$ then 
 $$0=pr_{2q}(\lambda(v,u)-u)=pr_{2q}(\log(\exp(pr_{(2q-2)}(v)+pr_{2q}(v)).\exp(u))-u)=pr_{2q}(dv)$$
 and due to $H^0(M,Der^{(2)}(\Lambda E))=0$  it is $pr_{2q}(v)=0$. Finally we have $pr_{2q}(\log(F_{2q}(0,u)))=pr_{2q}(pr_{(2q-2)}(u))=0$.
 
 \smallskip\noindent
 Secondly we check that (\ref{eqF}) does not contradict the derivative conditions $dD^\prime_{2q}(u)=R_{2q}(u)$ for  $q\geq 2$ and $d(pr_2(u))=0$.   Deriving (\ref{eqF}), this is equivalent to checking whether: 
 \begin{align}\label{zw}
  &pr_{2q}(\mathbf d \exp(pr_{(2q-2)}(u))-\mathbf d \exp(pr_{(2q-2)}(\lambda(v,u)))-d\log(F_{2q}(v,u)))=0
 \end{align}
 Now for reasons of degree:
 \begin{align*}
  pr_{(2q-2)}(\lambda(v,u))&=pr_{(2q-2)}(\log(\exp(v).\exp(u)))=pr_{(2q-2)}(\log(F_{2q}(v,u)))\\
  &=pr_{(2q)}(\log(F_{2q}(v,u)))-pr_{2q}(\log(F_{2q}(v,u)))
 \end{align*}
and:
 \begin{align*}
  pr_{(2q)}(\exp( pr_{(2q-2)}(\lambda(v,u))))=pr_{(2q)}(F_{2q}(v,u))-pr_{2q}(\log(F_{2q}(v,u)))
 \end{align*}
Using this, (\ref{zw}) is equivalent to:
 \begin{align*}
  &pr_{2q}(\mathbf d \exp(pr_{(2q-2)}(u))-\mathbf d F_{2q}(v,u))=0
 \end{align*}
This always holds since for reasons of degree:
  \begin{align*}
   \mathbf d F_{2q}(v,u)=pr_{(2q)}\mathbf d ( \exp(pr_{(2q-2)}(v)).\exp(pr_{(2q-2)}(u)))=pr_{(2q)}(\mathbf d \exp(pr_{(2q-2)}(u))
  \end{align*}

\end{proof}

\section{Constructing non-split supermanifolds from cochains of nilpotent derivations}
\noindent
We have seen that a necessary condition on an element $u\in C^1(M,Der^{(2)}(\Lambda E))$ for $\exp(u) \in Z^1(M,G_E)$ (i.e. $u\in \tilde C^1(M,Der^{(2)}(\Lambda E))$) is $R_{2q}(u)\in B^2(M, Der_{2q}(\Lambda E))$ for $q\geq 2$ and $d(pr_2(u))=0$. In particular, the weaker condition $R_{2q}(u)\in Z^2(M, End_{2q}(\Lambda E))$ for $q\geq 2$ has to be satisfied. For shortening the notation we denote $u=\sum_{k=1}^\infty u_{2k}$ with $u_{2k}\in C^1(M,Der_{2k}(\Lambda E))$. For $q=2$ we see using $(du_2)_{ijk}=u_{2,ij}+u_{2,jk}+u_{2,ki}=0$:
\begin{align}\label{eq:5}
 R_{4}(u)_{jkl}=&\frac{1}{2}(u_{2,jk}^2+u_{2,kl}^2+u_{2,lj}^2)+u_{2,jk}u_{2,kl}+u_{2,jk}u_{2,lj}+u_{2,kl}u_{2,lj}\\\nonumber
 =&\frac{1}{2}(du_2^2)_{jkl}+u_{2,lj}^2+u_{2,jk}u_{2,kl}+u_{2,jk}u_{2,lj}+u_{2,kl}u_{2,lj}=\frac{1}{2}(du_2^2)_{jkl}+u_{2,jk}u_{2,kl}
\end{align}
Hence again by $du_2=0$:
\begin{align*}
 (dR_4(u))_{ijkl}=u_{2,jk}u_{2,kl}-u_{2,ik}u_{2,kl}+u_{2,ij}u_{2,jl}-u_{2,ij}u_{2,jk}=0
\end{align*}
So $R_4(u)\in Z^2(M, End_{4}(\Lambda E))$ independently of the choice of $u_2\in Z^1(M,Der_{2}(\Lambda E))$. We now analyze the condition $R_6(u)\in Z^2(M, End_{4}(\Lambda E))$. Therefor we study $R_6(u_2)$ with $u_2\in Z^1(M,Der_{2}(\Lambda E))$ first. By direct calculation it is:
\begin{align}\nonumber
 R_{6}(u_2)_{ijk}=&\quad u_{2,ij}u_{2,jk}u_{2,ki}\\\nonumber&+\frac{1}{2}(u_{2,ij}u_{2,jk}^2+u_{2,ij}^2u_{2,jk}+u_{2,ij}u_{2,ki}^2+u_{2,ij}^2u_{2,ki}+u_{2,jk}u_{2,ki}^2+u_{2,jk}^2u_{2,ki})\\\nonumber&+\frac{1}{6}(u_{2,ij}^3+u_{2,jk}^3+u_{2,ki}^3)\\\label{eq:1}
 =&\quad u_{2,ij}u_{2,jk}u_{2,ki}+\frac{1}{2}[u_{2,ij},u_{2,jk}^2]-\frac{1}{3}(u_{2,ij}^3+u_{2,jk}^3+u_{2,ki}^3)
\end{align}
Further by direct calculation using $du_2=0$:
\begin{align}\label{eqR6}
R_{6}(u)_{ijk}=& R_{6}(u_2)_{ijk}+u_{2,ij}u_{4,jk}+u_{4,ij}u_{2,jk}+u_{2,ij}u_{4,ki}+u_{4,ij}u_{2,ki}+u_{2,jk}u_{4,ki}+u_{4,jk}u_{2,ki}\nonumber\\&+\frac{1}{2}(u_{2,ij}u_{4,ij}+u_{4,ij}u_{2,ij}+u_{2,jk}u_{4,jk}+u_{4,jk}u_{2,jk}+u_{2,ki}u_{4,ki}+u_{4,ki}u_{2,ki})\nonumber\\
=&R_{6}(u_2)_{ijk}+u_{4,ij}(\frac{1}{2}u_{2,ij}+u_{2,jk}+u_{2,ki})+\frac{1}{2}u_{2,ij}u_{4,ij}\nonumber\\
&+u_{4,jk}(\frac{1}{2}u_{2,jk}+u_{2,ki})+(\frac{1}{2}u_{2,jk}+u_{2,ij})u_{4,jk}\nonumber\\
&+\frac{1}{2}u_{4,ki}u_{2,ki}+(\frac{1}{2}u_{2,ki}+u_{2,ij}+u_{2,jk})u_{4,ki}\nonumber\\
=&R_{6}(u_2)_{ijk}+\frac{1}{2}([u_{2,ij},u_{4,ij}]+[u_{2,jk},u_{4,jk}]-[u_{2,ki},u_{4,ki}])+[u_{2,ij},u_{4,jk}]\nonumber\\
=&R_{6}(u_2)_{ijk}+\frac{1}{2}(d[u_{2},u_{4}])_{ijk}+[u_{2,ij},u_{4,jk}]
\end{align}
So it follows that $R_{6}(u)\in Z^2(M, End_6(\Lambda E))$ is equivalent to: 
\begin{align}\label{eq:2}
 d R_{6}(u_2)=-d([u_{2,ij},u_{4,jk}])_{ijk}
\end{align}
The left hand side is $(d R_{6}(u_2))_{ijkl}= R_{6}(u_2)_{jkl}- R_{6}(u_2)_{ikl}+ R_{6}(u_2)_{ijl}- R_{6}(u_2)_{ijk}$. The summand $\frac{1}{3}(u_{2,ij}^3+u_{2,jk}^3+u_{2,ki}^3)_{ijk}=\frac{1}{3}du_2^3$ in (\ref{eq:1}) has no contribution and we find with $du_2=0$ and $u_{2,jk}^2+u_{2,kl}^2-u_{2,jl}^2=(du_2^2)_{jkl}$:
\begin{align}\label{eq:3}
&(d R_{6}(u_2))_{ijkl}\\\nonumber=&u_{2,jk}u_{2,kl}u_{2,lj}-u_{2,ik}u_{2,kl}u_{2,li}+u_{2,ij}u_{2,jl}u_{2,li}-u_{2,ij}u_{2,jk}u_{2,ki}\\\nonumber&+\frac{1}{2}([u_{2,jk},u_{2,kl}^2]-[u_{2,ik},u_{2,kl}^2]+[u_{2,ij},u_{2,jl}^2]-[u_{2,ij},u_{2,jk}^2])\\\nonumber=&u_{2,jk}u_{2,kl}u_{2,lj}-u_{2,ik}u_{2,kl}u_{2,li}+u_{2,ij}u_{2,jl}u_{2,li}-u_{2,ij}u_{2,jk}u_{2,ki}-\frac{1}{2}[u_{2,ij},(du_2^2)_{jkl}]
\end{align}
For the right hand side of (\ref{eq:2}) we have with  $du_2=0$:
\begin{align*}
 -(d([u_{2,ij},u_{4,jk}])_{ijk})_{ijkl}
 =&-[u_{2,jk},u_{4,kl}]+[u_{2,ik},u_{4,kl}]-[u_{2,ij},u_{4,jl}]+[u_{2,ij},u_{4,jk}]\\\nonumber
 =& \quad \ [u_{2,ij},u_{4,jk}+u_{4,kl}+u_{4,lj}]=[u_{2,ij},(du_4)_{jkl}] 
\end{align*}
Under the stronger necessary condition and $du_4=- R_{4}(u)$ we obtain:
\begin{align}\label{eq:4}
 -(d([u_{2,ij},u_{4,jk}])_{ijk})_{ijkl}=-[u_{2,ij},R_{4}(u)_{jkl}] 
\end{align}
Inserting  (\ref{eq:3}), (\ref{eq:4}), and (\ref{eq:5}) in (\ref{eq:2}), $R_{6}(u)\in Z^2(M, End_6(\Lambda E))$ is equivalent to:
\begin{align*}
u_{2,jk}u_{2,kl}u_{2,lj}-u_{2,ik}u_{2,kl}u_{2,li}+u_{2,ij}u_{2,jl}u_{2,li}-u_{2,ij}u_{2,jk}u_{2,ki}+[u_{2,ij},u_{2,jk}u_{2,kl}]=0
\end{align*}
By $du_2=0$ this is always satisfied. Hence we summarize:

\begin{prop}\label{P}
Let $u\in C^1(M,Der^{(2)}(\Lambda E))$. The following implications of necessary conditions for $\exp(u)\in Z^1(M,G_E)$ exist:\smallskip\\
If  $du_2=0$ is satisfied then $R_4(u)\in Z^2(M,End_4(\Lambda E))$ is satisfied. \smallskip\\
If $du_2=0$ and $du_4=-R_4(u)$ are satisfied then $R_6(u)\in Z^2(M,End_6(\Lambda E))$ is satisfied and: $$R_6(u)\in Z^2(M,Der_6(\Lambda E))\qquad \Leftrightarrow\qquad R_6(u_2)\in C^2(M,Der_6(\Lambda E)) $$
\end{prop}

\begin{proof}
 The second part of the second statement follows from (\ref{eqR6}).
\end{proof}

In the following we denote: $$\begin{array}{rcl}\tilde Z^1(M,Der_2(\Lambda E)):=\big\{u_2\in Z^1(M,Der_2(\Lambda E)) \ \big| && R_{4}(u_2)\in C^2(M,Der_4(\Lambda E))  \\&\mbox{ and }& R_{6}(u_2)\in C^2(M,Der_6(\Lambda E))\ \ \big\}\end{array}$$ 
It follows with Corollary \ref{cor}:
\begin{cor}\label{corx}
 Let $E\to M$ be a vector bundle of rank $6$ or $7$. Assume further that $H^2(M,Der_4(\Lambda E))=H^2(M,Der_6(\Lambda E))=0$.\smallskip\\ The necessary and sufficient condition on a cochain $u_2\in C^1(M,Der_2(\Lambda E))$ for the existence of a  $u\in \tilde C^1(M,Der^{(2)}(\Lambda E))$ with $pr_2(u)=u_2$  is  $u_2\in \tilde Z^1(M,Der_2(\Lambda E))$.\smallskip\\ If in addition $H^0(M,Der^{(2)}(E))=0$ then there is a $D$ such that:
 $$\sigma_D:H^1(M,G_E)\to\frac{\tilde Z^1(M,Der_2(\Lambda E))}{B^1(M,Der_2(\Lambda E))}\oplus H^1(M,Der_4(\Lambda E))\oplus H^1(M,Der_6(\Lambda E))$$
 is a well-defined bijection.
\end{cor}

\begin{proof}
 By the proposition, $R_{4}(u_2) \in  Z^1(M,Der_4(\Lambda E))$. Fix any cochain $u_4$ such that $R_{4}(u_2)=-du_4$. Now $R_{6}(u_2+u_4) \in  Z^1(M,Der_6(\Lambda E))$ is satisfied so there is a $u_6$ with $R_{6}(u_2+u_4)=-du_6$.
\end{proof}

\begin{example}
 Let $M=\mathbb P^2(\mathbb C)\backslash \{[0:0:1]\}$ and $E$ any vector bundle of rank up to $7$. Since $M$ is strongly $2$-complete, we have $H^2(M,Der^{(2)}(\Lambda E))=0$ . There is a Leray cover (for coherent sheaf cohomology) of $M$ with two components: two of the three standard coordinate charts of $\mathbb P^2(\mathbb C)$. So we have $R_{2q}\equiv 0$ and $\tilde Z^1(M,Der_2(\Lambda E))=Z^1(M,Der_2(\Lambda E))$ and by Corollary \ref{corx} any $u_2\in Z^1(M,Der_2(\Lambda E))$ can be continued to a $u\in C^1(M,Der^{(2)}(\Lambda E))$ with $pr_2(u)=u_2$ such  that $\exp(u)$ defines a supermanifold structure on $M$.
\end{example}

For any vector bundle $E$ and $q\geq 1$ we have the short exact sequence:
$$ 0\to Hom_{\mathcal O_M}(E,\Lambda^{2q+1}E)\to Der_{2q}(\Lambda E)\to Der(\mathcal O_M,\Lambda^{2q}E)\to 0$$
the second arrow by continuation via graded Leibniz rule trivially on $\mathcal O_M$ and the third arrow by restriction. By the long exact sequence of cohomology we obtain exactness of:
\begin{align}\label{dfg}
 H^0(M,Hom_{\mathcal O_M}(E,\Lambda^{2q+1}E))\to H^0(M,Der_{2q}(\Lambda E))\to H^0(M,Der(\mathcal O_M,\Lambda^{2q}E))
\end{align}

\begin{example}
The conditions $H^2(M,Der_4(\Lambda E))=H^2(M,Der_6(\Lambda E))=0$ are satisfied for compact Riemannian surfaces $M$.  
Let $M=\mathbb P^1(\mathbb C)$ and fix a sum of line bundles $E=\bigoplus_{i=1} ^{k}\mathcal O(l_i)$ with $k\leq7$ and $l_1\leq \cdots \leq l_k$ such that:
\begin{align*}
 l_{k-1}+l_k<-2  \qquad \mbox{ and} \qquad 
 l_{k-2}+l_{k-1}+l_k-l_1<0 
\end{align*}
Then  for $q=1,2,3$:
\begin{align*}
 &Hom_{\mathcal O_M}(E,\Lambda^{2q+1}E)\cong \mathcal O_{\Lambda^{2q+1}E}\otimes \mathcal O_{E^\ast}\\
 &Der(\mathcal O_M,\Lambda^{2q}E)\cong \mathcal O_{\Lambda^{2q}E}\otimes \mathcal O(2)
\end{align*}
Now due to $H^0(M,\mathcal O(l))=0$ for $l<0$: $$H^0(M,Hom_{\mathcal O_M}(E,\Lambda^{2q+1}E))=H^0(M,Der(\mathcal O_M,\Lambda^{2q}E))=0$$ By (\ref{dfg}), $H^0(M,Der^{(2)}(\Lambda E))=0$. There is a Leray cover of $M$ (for coherent sheaf cohomology) with two components, so $\tilde Z^1(M,Der_2(\Lambda E))=Z^1(M,Der_2(\Lambda E))$. Hence Corollary \ref{corx} yields a bijection $\sigma_D:H^1(M,G_E)\to H^1(M,Der^{(2)}(\Lambda E))$. 
\end{example}

\bibliographystyle{alpha} 

\addcontentsline{toc}{section}{References}
\begin{thebibliography}{9}
\bibitem[Gr82]{Gr} Green, P., \emph{On holomorphic graded manifolds},  Proc. Amer. Math. Soc.  85  (1982), no. 4, pp. 587-590
\bibitem[Ka14]{Ka} Kalus, M., \emph{Complex supermanifolds of low odd dimension and the example of the complex projective line},  arXiv:1405.5065
\bibitem[On99]{On} Onishchik, A.L., \emph{On the classification of complex analytic supermanifolds},   Lobachevskii J. Math.  4  (1999), pp. 47-70
\bibitem[Ro82]{Ro} Rothstein, M.J., \emph{Deformations of complex supermanifolds},  Proc. Amer. Math. Soc.  95  (1985),  no. 2, pp. 255-260
\end{thebibliography}

\end{document}